\documentclass[a4paper,reqno]{amsart}

\usepackage{color}
\usepackage{amsthm}
\usepackage{changepage}
\usepackage{amssymb}
\usepackage{mathdots}
\usepackage{amsmath}
\usepackage{amscd}
\usepackage{amsopn}
\usepackage{enumerate}
\usepackage{mathtools}
\usepackage{url}

\usepackage{hyperref}\hypersetup{colorlinks}

\usepackage{color} 

\definecolor{darkred}{rgb}{1,0,0} 
\definecolor{darkgreen}{rgb}{0,0.8,0}
\definecolor{darkblue}{rgb}{0,0,1}

\hypersetup{colorlinks,
linkcolor=darkblue,
filecolor=darkgreen,
urlcolor=darkred,
citecolor=darkgreen}

\numberwithin{equation}{section}

\newtheorem {Theorem}{Theorem}

\numberwithin{Theorem}{section}

\newtheorem {Lemma}[Theorem]    {Lemma}

\newtheorem {Proposition}[Theorem]{Proposition}

\theoremstyle{definition}

\newtheoremstyle{TheoremForIntro} 
        {.6em}{.6em}              
        {\itshape}                      
        {}                              
        {\bfseries}                     
        {. }                             
        { }                             
        {\thmname{#1}\thmnote{ \bfseries #3}}
    \theoremstyle{TheoremForIntro}

\expandafter\chardef\csname pre amssym.def at\endcsname=\the\catcode`\@
\catcode`\@=11
\def\undefine#1{\let#1\undefined}
\def\newsymbol#1#2#3#4#5{\let\next@\relax
 \ifnum#2=\@ne\let\next@\msafam@\else
 \ifnum#2=\tw@\let\next@\msbfam@\fi\fi
 \mathchardef#1="#3\next@#4#5}
\def\mathhexbox@#1#2#3{\relax
 \ifmmode\mathpalette{}{\m@th\mathchar"#1#2#3}%
 \else\leavevmode\hbox{$\m@th\mathchar"#1#2#3$}\fi}
\def\hexnumber@#1{\ifcase#1 0\or 1\or 2\or 3\or 4\or 5\or 6\or 7\or 8\or
 9\or A\or B\or C\or D\or E\or F\fi}

\font\teneufm=eufm10
\font\seveneufm=eufm7
\font\fiveeufm=eufm5
\newfam\eufmfam
\textfont\eufmfam=\teneufm
\scriptfont\eufmfam=\seveneufm
\scriptscriptfont\eufmfam=\fiveeufm

\catcode`\@=\csname pre amssym.def at\endcsname

\newcommand{\g}{{\mathcal g}}

\def    \HH      {{\mathbb H}}

\def    \C      {{\mathbb C}}
\def    \R      {{\mathbb R}}
\def    \Z      {{\mathbb Z}}
\def    \M      {{\mathbb M}}

\def    \P    {{\mathbb P}}

\def    \tr     {\operatorname{tr}}

\newcommand{\defin}[1]{{{\itshape #1}}}

\definecolor{NoteColor}{rgb}{1,0,0}

\copyrightinfo{2015}{Daniele Alessandrini and Qiongling Li}

\author{Daniele Alessandrini}
\address{Daniele Alessandrini, Universitaet Heidelberg, INF 288, 69120, Heidelberg, Germany}
\email{daniele.alessandrini@gmail.com}
\thanks{This work was started when both authors were visiting MSRI, research at MSRI is supported in part by NSF grant DMS-0441170. Both authors acknowledge the support from U.S. National Science Foundation grants DMS 1107452, 1107263, 1107367 ``RNMS: GEometric structures And Representation varieties'' (the GEAR Network). 
}

\author{Qiongling Li}
\address{Qiongling Li, Aarhus University, Ny Munkegade 118, 8000 Aarhus C, Denmark}
\email{qiongling.li@gmail.com}


\title{AdS 3-manifolds and Higgs bundles}

\date{\today}

\begin{document}

\begin{abstract}
In this paper we investigate the relationships between closed AdS $3$-manifolds and Higgs bundles. We have a new way to construct AdS structures that allows us to see many of their properties explicitly, for example we can recover the very recent formula by Tholozan for the volumes. 

We also find applications to the theory of minimal immersions into quadrics with their natural pseudo-Riemannian structure: using the geometry of the AdS manifolds we can characterize the representations admitting equivariant minimal immersions of the Poincar\'e disc into the Klein quadric, the Grassmannian $\text{Gr}(2,4)$, and understand the geometry of these minimal immersions. 
\end{abstract}

\maketitle

\section{Introduction}

The theory of Higgs bundles, founded by Hitchin \cite{HitchinSelfDuality} and Donaldson \cite{DonaldsonTwisted}, and generalized in \cite{SimpsonHBandLocalSys} and \cite{Corlette}, has been very useful in the study of representations of surface groups. This is due to Hitchin's equations, a system of global elliptic PDEs that gives homeomorphisms between the moduli spaces of Higgs bundles and character varieties, providing a powerful tool for the study of the global structure of these spaces, e.g. determining the number of connected components and their topology. More difficult is to get information about a single representation from the corresponding Higgs bundle, since solutions of Hitchin's equations are very non-explicit. The problem of giving a geometric interpretation to the solutions of Hitchin's equations is one of the motivations for this paper.

In this paper we investigate the relationships between Higgs bundles and geometric structures on $3$-manifolds. These relationships were explored in Baraglia's thesis \cite{BaragliaThesis}, and more recently in the present work, in some other work in progress of the authors and of the authors with Brian Collier. A geometric structure is determined by a developing pair, a pair consisting of a representation and a developing map (see section \ref{sec:geometricstructures}). The Higgs bundle encodes the representation in the usual way, we only need to describe the developing map in terms of the Higgs bundle data and the solutions of Hitchin's equations. 

The geometric structures we study in this paper are the $3$-dimensional Anti-de Sitter structures (abbreviated in AdS structures), Lorentz metrics with constant curvature $-1$. They initially arose as models for general relativity, but their rich mathematical theory makes them very interesting geometric objects, independently on physical applications. The theory of closed and open AdS $3$-manifolds are quite different, here we will deal only with closed AdS $3$-manifolds.  

The AdS $3$-manifolds are locally modeled on the Anti-de Sitter $3$-space $\M$ (see section \ref{sec:AdS-geometry}), whose isometry group is isomorphic to $O(2,2)$. If $\rho_1,\rho_2$ are two representations of the fundamental group of a closed surface $S$ in $SL(2,\R)$, their tensor product $\rho_1 \otimes \rho_2$ is a representation in $O(2,2)$. We will use the theory of Higgs bundles for $SL(2,\R)$ to describe the tensor product $\rho_1 \otimes \rho_2$, and to construct a circle bundle $U$ over $S$. The topology of this circle bundle is explicitly determined in terms of the representations $\rho_1,\rho_2$. 

If the two representations satisfy a condition of domination (see section \ref{sec:domination}), we will construct an AdS structure on $U$ with holonomy $\rho_1\otimes \rho_2$, see section \ref{sec:construction}. This gives a new proof of a fundamental theorem originally proved by Salein in \cite{SaleinExotiques}, and later reproved in \cite{GueritaudKassel}.

In our proof, the structure we construct has a natural parametrization, given in terms of the solutions of Hitchin's equations, and some properties of the AdS manifold can be seen easily using this parametrization. We can determine the underlying topology of the AdS manifold, we can see that the fibers of the circle bundle $U$ are time-like geodesics for the AdS structure (see section \ref{sec:fibers}), and that the manifold is the quotient of the model space $\M$ by the representation $\rho_1\otimes \rho_2$, thus obtaining that this representation acts freely and properly discontinuously on $\M$ (see section \ref{sec:discontinuous}). 

We can also easily compute the volume of the AdS structure with an explicit formula that shows that the volume only depends on the Euler numbers of the representations $\rho_1,\rho_2$. This result was cited as Question 2.3 in the list of open problems \cite{BBDGGKKSZ}. It was first answered some months ago in Nicolas Tholozan's thesis \cite{TholozanThese}, see also \cite{TholozanVolumes}. Labourie then related the volumes of AdS $3$-manifolds with the Chern-Simons invariants. In this paper we give a new proof of the formula, see section \ref{sec:volumes}.  

In this way, we can recover much of the theory of closed AdS $3$-manifolds, with new methods based mainly on Higgs bundles, harmonic maps and study of solutions of Hitchin's equations.

As an original result for this paper, given an AdS manifold with holonomy $\rho_1\otimes \rho_2$, we can use its parametrization  as a circle bundle, and the AdS structure on the circle fibers to construct $\rho_1\otimes \rho_2$-equivariant minimal immersions of the hyperbolic plane in the Klein quadric, the Grassmannian $\text{Gr}(2,4)$ of $2$-planes in $\R^4$, equipped with its natural pseudo-Riemannian metric. This gives a characterization of the representations in $O(2,2)$ that admit equivariant minimal surfaces in the Klein quadric.

We can also characterize the conformal structure of the minimal surfaces we construct, this is the conformal structure with vanishing Pfaffian found by Tholozan in \cite{TholozanDeforming}. This part is in section \ref{sec:fibers}, Theorems \ref{thm:harmimm} and \ref{thm:minsurf}. 

To obtain our results, we also complete the description of the Higgs bundles for $SL(2,\R)$, see theorems \ref{thm:splitting} and \ref{metricandvolume}.

At first glance, it seems we just construct and study some examples of closed AdS $3$-manifolds. We would like to remark that these examples are actually very general, by understanding them we can understand all closed AdS $3$-manifolds: given a closed AdS $3$-manifold $Y$, there exist representations $\rho_1, \rho_2$ of a surface group in $SL(2,\R)$ with the property that $\rho_1$ strictly dominates $\rho_2$, such that $Y$ and the quotient $\M / \rho_1\otimes \rho_2$ have a common finite covering. In this way it is possible to recover most properties of $Y$ from the properties of $\M / \rho_1\otimes \rho_2$, for example it is possible to compute the volume of $Y$. This property of the common finite covering comes from the results of Klingler \cite{KlinglerCompletude}, Kulkarni and Raymond \cite{KulkarniRaymond} and Kassel \cite{KasselThese}.

\subsection*{Acknowledgments} We need to thank many people for useful conversations and comments about this topic. In particular we wish to thank Brian Collier, Bill Goldman, Fanny Kassel, Nicolas Tholozan, Anna Wienhard and Mike Wolf.

\section{Anti-de Sitter 3-manifolds} \label{sec:AdS-geometry}  \label{sec:sketch-construction}  \label{sec:geometricstructures}

Given a $3$-manifold $N$, an \defin{Anti-de Sitter structure} on $N$, or briefly an \defin{AdS structure} on $N$, is a Lorentz metric of constant curvature $-1$. Recall that a \defin{Lorentz metric} is a pseudo-Riemannian metric of signature $(2,1)$. 

Every AdS structure is locally isometric to a model space, the \defin{Anti-de Sitter space}, that can be described explicitly in the following way. 

Let $Q$ denote a non-degenerate symmetric bilinear form on $\R^4$ with signature $(2,2)$. The \defin{Anti-de Sitter space} is identified with the unit sphere of $Q$:
$$\M = \{ x \in \R^4 \ |\ Q(x,x) = 1 \} $$
For every point $x \in \M$, the tangent space to $\M$ at $x$ is given by $x^{\perp Q} $, the orthogonal with reference to $Q$. The bilinear form $Q$, restricted to every tangent space $T_x \M = x^{\perp Q} $, gives a bilinear form of signature $(2,1)$. This defines the Lorentz metric on the model space $\M$. With our conventions, a tangent vector $v\in T_x \M$ is \defin{time-like} if $Q(v,v)>0$, \defin{space-like} if $Q(v,v)<0$, and \defin{light-like} if $Q(v,v)=0$. 

The group of isometries of $\M$ is the orthogonal group preserving $Q$, here denoted by $O(2,2)$. Here we will mainly use the subgroup $SO_0(2,2)$, the connected component of the identity of $O(2,2)$. The subgroup $SO_0(2,2)$, the group of isometries preserving both space and time orientation, can be identified with $SL(2,\R) \times SL(2,\R) / \Z_2$ in the following way. Let $\omega$ a volume form on $\R^2$. The tensor product $\R^2 \otimes \R^2$ inherits the bilinear form $Q = \omega\otimes\omega$, this form is symmetric and with signature $(2,2)$. 
Given matrices $A,B\in SL(2,\R)$, they act on $\R^2$ preserving $\omega$, hence their tensor product $A\otimes B$ acts on $\R^2 \otimes \R^2$ preserving $Q$. If $A=B=-\mbox{Id}$, the action is trivial. 

Given a $3$-manifold $N$, we want to construct AdS structures on $N$. We will use the model space $\M$ and mirror its local geometry on $N$. More precisely, $N$ is covered by open charts $\{U_i\}$ with associated diffeomorphism $\phi_i$ from $U_i$ to an open subset of $\M$. Whenever two open sets $U_i,U_j$ intersect, the transition functions have to be locally the restrictions of elements of $O(2,2)$. If the transition functions lie in $SO_0(2,2)$, the structure has space and time orientation. 

This construction comes from the theory of geometric structures, or $(G,X)$ structures, where $G$ is a Lie group acting transitively and effectively on a manifold $X$. In our case, $G = SO_0(2,2)$, and $X=\M$. For the general theory of geometric structures, the reader is referred to \cite{GoldmanSurvey} or \cite{ThurstonBook}. A \defin{geometric structure} is defined as a maximal atlas of charts ${\{(U_i,\phi_i)\}}_{i\in I}$ satisfying the properties given above.

An equivalent way to see a geometric structure is via the \defin{development pair}, $(D,h)$, where $h:\pi_1(N) \rightarrow G$ is a representation called \defin{holonomy representation}, and $D:\widetilde{N} \rightarrow X$ is a $\rho$-equivariant local diffeomorphism, called \defin{development map}. 

This paper was inspired by the following question: given a representation $\rho:\pi_1(N) \rightarrow G$, how can we construct a geometric structure on $N$ with holonomy $\rho$?

To address this question, we can use another equivalent way to see a geometric structure, called the \defin{graph} of a geometric structure (see \cite{GoldmanSurvey} for details).  

Given a representation $\rho:\pi_1(N) \rightarrow G$, and an effective action of $G$ on $X$, we can construct a unique flat bundle $p:B \rightarrow N$ over $N$ with fiber $X$, structure group $G$, and monodromy equal to $\rho$. 
  
The total space $B$ is given by 
$$ B = X_\rho = \left(\widetilde{N} \times X\right)/ \pi_1(N) $$
where the group $\pi_1(N)$ acts diagonally on the product, with the natural action on $\widetilde{N}$ and with the representation $\rho$ on $X$. The space $X_\rho$ has a natural map to $N = \widetilde{N}/\pi_1(N)$ that is a fiber bundle. The flat structure on the product $\widetilde{N} \times X$ descends to the quotient. The monodromy of the flat bundle $B = X_\rho$ is exactly the representation $\rho$, so the flat bundle $X_\rho$ encodes the representation.

To construct a geometric structure with holonomy $\rho$, we only need to add one more piece of information, the developing map. We need to interpret the developing map in the language of flat bundles. 

The notion of a $\rho$-equivariant map can be translated very naturally: it is a section of the bundle $X_\rho$. Given a section $s:N \rightarrow X_\rho$, it can be lifted to the universal covering $\widetilde{s}:\widetilde{N} \rightarrow \widetilde{X_\rho}$, where  $\widetilde{X_\rho}$ is the pull back of $X_\rho$ to $\widetilde{N}$. Since $\widetilde{X_\rho} = \widetilde{N} \times X$, the projection on the second component gives a $\rho$-equivariant map $\widetilde{N} \rightarrow X$. By reversing the construction, every $\rho$-equivariant map comes from a section. 

Following \cite{GoldmanSurvey}, we call \defin{transverse section} a section whose associated $\rho$-equivariant map is a local diffeomorphism, i.e. a developing map of a geometric structure.

In this paper we want to use the technique of graph of geometric structures to construct AdS structures with prescribed holonomy.

Let $S$ be a closed surface, and $\rho_1,\rho_2:\pi_1(S) \rightarrow SL(2,\R)$ be two reductive representations. Their tensor product $\rho=\rho_1 \otimes \rho_2$ is a representation in $SO_0(2,2)$ by the isomorphism $SO_0(2,2) = (SL(2,\R) \times SL(2,\R)) / \Z_2$ explained in section \ref{sec:AdS-geometry}. Let $p:U \rightarrow S$ be a circle bundle over $S$. The representation $\rho$ induces 
$$\bar{\rho} = \rho \circ p_*:\pi_1(U) \rightarrow SO_0(2,2),$$
a representation of $\pi_1(U)$ that is trivial on the fibers. We want to construct an AdS structure on the 3-manifold $U$ with holonomy $\bar{\rho}$. 

We denote by $E$ the vector bundle $(\R^4)_\rho \rightarrow S$. Since $\rho$ is a representation in $SO_0(2,2)$, the bundle $E$ is naturally equipped with a symmetric bilinear form $Q$ of signature $(2,2)$ on every fiber. The fiber bundle $\M_\rho \rightarrow S$ is a sub-bundle of the vector bundle $E$, whose fibers are defined by the equation $Q(v,v)=1$. To construct an AdS structure on $U$, we need to consider the pull-back bundles $p^*E \rightarrow U$ and $p^*\M_\rho = \M_{\bar{\rho}} \rightarrow U$, and find a transverse section $s:U \rightarrow \M_{\bar{\rho}}$.

It will not be possible to achieve this for every representation $\rho_1, \rho_2$ of $\pi_1(S)$, and for every circle bundle $U$, we need to add some hypothesis. To choose a suitable circle bundle $U$, we need to understand the structure of the bundle $E \rightarrow S$. Denote by $e_1, e_2 \in [2-2g,2g-2]\cap 2\Z$ the Euler numbers of $\rho_1, \rho_2$ respectively. We will see in the next section that the bundle $E$ can be split as a direct sum of $2$ vector bundles of rank $2$, $E = F_1 \oplus F_2$ such that $F_1$ has Euler class $|e_2-e_1|$ and $F_2$ has Euler class $|e_1+e_2|$. The two sub-bundles can be chosen such that $F_1 \perp_Q F_2$, and such that $F_1$ is time-like and $F_2$ is space-like. 

We will choose $U$ as the circle bundle 
$U = \{ v \in F_1 \ |\ Q(v,v) = 1 \}. $
This is the unit part of $F_1$, hence a circle bundle with Euler class $|e_2-e_1|$. 
We can construct a very natural section $s: U \rightarrow \M_{\bar{\rho}}$, the tautological section that associates to every point $v$ of $U$ the same point $v$ seen as a point of $\M_{\bar{\rho}}$.

The last thing needed to construct the AdS structure on $U$ is to verify the transversality condition of the section $s$. Since $E$ is a vector bundle, its flat structure is described by a flat connection $\nabla$, and the transversality condition can be formulated in terms of the derivatives of  $s$ with reference to  $\nabla$.  

To choose the sub-bundles $F_1$, $F_2$ in the right position with reference to $\nabla$, and to compute the derivatives of $s$, we need some special coordinates on the surface $S$ and the bundle $E$ that make possible to write $\nabla$ explicitly enough to do computations. For this we will use the tool given by Higgs bundles.

\section{Higgs bundles and flat bundles} \label{sec:higgsandflat}

The two $SL(2,\R)$ representations $\rho_1,\rho_2$ of the previous section, with Euler numbers $e_1, e_2$, correspond to flat unimodular vector bundles $(E_1,\nabla_1,\omega_1), (E_2,\nabla_2,\omega_2)$, where the $\nabla_i$ are flat connections and the $\omega_i$ are $\nabla_i$-parallel volume forms. In a local frame, we will write $\nabla_i = d + A_i$, where $A_i$ is the connection form. 

The representation $\rho = \rho_1 \otimes \rho_2$ corresponds to the flat bundle $(E,\nabla,Q)$, where $E = E_1 \otimes E_2$, $Q = \omega_1 \otimes \omega_2$, and the connection can be described, in the compatible tensor frame, as $\nabla = d + A_1 \otimes \mbox{Id} + \mbox{Id} \otimes A_2$.  The bilinear form $Q$ is symmetric with signature $(2,2)$ (see section \ref{sec:AdS-geometry}) and it is parallel for $\nabla$.

We will also consider the corresponding complexified bundles $(E_i^\C,\nabla_i^\C,\omega_i^\C)$ and $(E^\C,\nabla^\C,Q^\C)$. Corresponding to the embeddings $E_i \subset E_i^\C$ and $E \subset E^\C$, there are $\C$-anti-linear involutions that we will denote by $\tau_i:E_i^\C \rightarrow E_i^\C$ and $\tau:E^\C \rightarrow E^\C$, such that $E_i = \{v \in E_i^\C \ |\ \tau_i(v)=v\}$ and $E = \{v \in E^\C \ |\ \tau(v)=v\}$. The involution is called the \defin{real structure}. Hence the full structure on the complexified bundles is $(E_i^\C,\nabla_i^\C,\omega_i^\C,\tau_i)$ and $(E^\C,\nabla^\C,Q^\C,\tau)$.

To describe more explicitly the bundles $E_i^\C,E^\C$, we will use the theory of Higgs bundles.
Choose a holomorphic structure $\Sigma$ on the surface $S$ and denote by $K$ its canonical bundle.
There is a unique $\rho_i$-equivariant harmonic map $\psi_i:\widetilde{\Sigma} \rightarrow \HH$ to the hyperbolic plane $\HH=SL(2,\R)/SO(2) \subset SL(2,\C)/SU(2)$, the space of all Hermitian metrics on $\C^2$. Hence, the $\rho_i$-equivariant map $\psi_i$ corresponds to a Hermitian metric $H_i$ that, together with $\nabla_i^\C$,
defines a holomorphic structure on $E_i^\C$, and holomorphic Higgs fields $\phi_i \in H^0(\Sigma,\mbox{End}(E_i^\C)\otimes K)$. The bundle $E_i^\C$ splits as a direct sum of a holomorphic line bundle and its inverse:
$$E_1^\C = L\oplus L^{-1}, \hspace{2cm} E_2^\C = N\oplus N^{-1}   $$
where $\deg L = \frac{1}{2} e_1$ and  $\deg N = \frac{1}{2} e_2$ (recall that $e_1,e_2$ are even numbers). 

The Higgs fields can be expressed in terms of this splitting as:
$$\phi_1 = \begin{pmatrix}0&\alpha\\ \beta &0\end{pmatrix}, \hspace{2cm} \phi_2 = \begin{pmatrix}0&\gamma\\\delta&0\end{pmatrix}$$
where $\alpha\in H^0(\Sigma, L^2K)$, $\beta\in H^0(\Sigma, L^{-2}K)$, $\gamma\in H^0(\Sigma, N^2K)$ and $\delta\in H^0(\Sigma, N^{-2}K)$. 

\begin{Theorem} \label{thm:splitting}
In terms of the splitting of $E_i$ as direct sum of two line bundles, the metric $H_i$ is diagonal, i.e. it can be written as
$$H_1=\begin{pmatrix}k^{-1}&0\\0&k\end{pmatrix} \hspace{2cm} H_2=\begin{pmatrix}h^{-1}&0\\0&h\end{pmatrix}$$ 
where $k\in \Gamma(\Sigma, \bar L \otimes L)$ and $h\in \Gamma(\Sigma, \bar N \otimes N)$.
\end{Theorem}
\begin{proof}
Higgs bundles for $SL(2,\R)$  have an $SO(2,\C)$ structure coming from the pairing between $L$ (or $N$) and $L^{-1}$ (or $N^{-1}$). In terms of the splitting, this structure is the $\C$-bilinear form 
$B_i = \begin{pmatrix}0&1\\1&0\end{pmatrix} $.
Consider the two involutions $A \rightarrow H_i^{-1} (\bar{A}^T)^{-1} H_i$, whose fixed points are the unitary matrices, and $A\rightarrow B_i^{-1} (A^T)^{-1} B_i$, whose fixed points are the orthogonal matrices. These involutions commute, hence we have the equation $B^{-1} H_i^T =H_i^{-1} \bar{B}^T $. This equation and the condition $\det(H_i) = 1$ implies the statement.
\end{proof}


The flat $SL(2,\R)$-connection $\nabla_i^\C$ is given explicitly by
$$\nabla_i^\C = d + A_i = d + H_i^{-1}\partial H_i + \phi_i + \phi_i^{*_{H_i}}. $$

The real structure $\tau_i:E_i^\C \rightarrow E_i^\C$ is given by the composition of the two commuting involutions that appeared in the proof of theorem \ref{thm:splitting}, explicitly:
$$E_1^\C \ni v=\begin{pmatrix}v_1\\v_2\end{pmatrix} \rightarrow  \begin{pmatrix}0 & k\\k^{-1}&0\end{pmatrix}\bar v = \begin{pmatrix}k \bar v_2\\ k^{-1}\bar v_1\end{pmatrix} \in E_1^\C $$
$$ E_2^\C \ni v=\begin{pmatrix}v_1\\v_2\end{pmatrix} \rightarrow \begin{pmatrix}0 & h\\h^{-1}&0\end{pmatrix}\bar v= \begin{pmatrix}h \bar v_2\\ h^{-1}\bar v_1\end{pmatrix} \in E_2^\C$$
Finally, the volume form $\omega_i^\C$ is given by 
$\omega_i^\C = \frac{i}{\sqrt{2}}\begin{pmatrix}0 & -1 \\ 1 & 0 \end{pmatrix}. $

This completes the description of the bundle $(E_i^\C,\nabla_i^\C,\omega_i^\C,\tau_i)$ in the special coordinates given by the choice of $\Sigma$. To describe the bundle $(E^\C,\nabla^\C,Q^\C,\tau)$ we just need to consider the tensor product of Higgs bundles, described in \cite{SimpsonHBandLocalSys}:
$$ E^\C = E_1^\C \otimes E_2^\C = LN\oplus LN^{-1} \oplus L^{-1}N\oplus L^{-1}N^{-1} $$
The Higgs field is given by
$$\Phi=\phi_1\otimes \mbox{Id}+\mbox{Id}\otimes\phi_2= \begin{pmatrix}0&\gamma&\alpha&0\\\delta&0&0&\alpha\\\beta&0&0&\gamma\\0&\beta&\delta&0\end{pmatrix} $$

The characteristic polynomial of $\Phi$ is 
$P(\Phi) = x^4 -2(\alpha\beta + \gamma \delta)x^2+(\alpha\beta -\gamma\delta)^2 $.
The second coefficient $-2(\alpha\beta + \gamma \delta)$ can be interpreted as the Hopf differential of the associated harmonic map. The square root of the fourth coefficient $\alpha\beta -\gamma\delta$ is usually called the \defin{Pfaffian}. It will be important in section \ref{sec:domination} and \ref{sec:fibers}.

The Hermitian metric solution of Hitchin's equations is just the tensor product 
$$H = H_1 \otimes H_2 =\text{diag}(\ h^{-1}k^{-1},hk^{-1}, h^{-1}k ,hk\ ).$$

The flat $SO_0(2,2)$ connection is given by
$\nabla^\C = d + H^{-1}\partial H + \Phi + \Phi^{*_H}. $

When taking derivations, we will derive in the complex directions $\tfrac{\partial}{\partial z}$ and $\tfrac{\partial}{\partial \bar z}$:
$$\nabla^\C_{\frac{\partial}{\partial z}} = \partial + H^{-1}\partial H + \Phi =$$
$$ = \partial + \begin{pmatrix}
\partial \log(h^{-1}k^{-1}) &\gamma                 &\alpha                 &0\\
\delta                      &\partial \log(hk^{-1}) &0                      &\alpha\\
\beta                       &0                      & \partial \log(h^{-1}k)&\gamma\\
0                           &\beta                  &\delta                 & \partial \log(hk)
\end{pmatrix}$$

$$\nabla^\C_{\frac{\partial}{\partial \bar z}}  = \bar \partial + \Phi^{*_H} = \bar \partial + \begin{pmatrix}0&h^2\bar\delta&k^2\bar\beta&0\\h^{-2}\bar\gamma&0&0&k^2\bar\beta\\k^{-2}\bar\alpha&0&0&h^2\bar\delta\\0&k^{-2}\bar\alpha&h^{-2}\bar\gamma&0\end{pmatrix}$$

The bilinear form and the real structure are given by:
$$ Q^\C = \omega_1 \otimes \omega_2 = \frac{1}{2}\begin{pmatrix} 0&0&0&-1\\0&0&1&0\\0&1&0&0\\-1&0&0&0 \end{pmatrix}$$
$$\tau:E^\C \ni v = \begin{pmatrix}v_1\\v_2\\v_3\\v_4\end{pmatrix}\rightarrow \begin{pmatrix}0&0&0&hk\\0&0&h^{-1}k&0\\0&hk^{-1}&0&0\\h^{-1}k^{-1}&0&0&0\end{pmatrix}\bar v = \begin{pmatrix}hk\bar v_4\\ h^{-1}k\bar v_3\\ hk^{-1}\bar v_2\\ h^{-1}k^{-1}\bar v_1\end{pmatrix} \in E^\C    $$

The real structure $\tau$ leaves invariant the two sub-bundles $LN\oplus L^{-1}N^{-1}$ and   $LN^{-1} \oplus L^{-1}N$. Denote by $F_1$ the real part of $LN^{-1} \oplus L^{-1}N$, and by $F_2$ the real part of $LN\oplus L^{-1}N^{-1}$, these are the real sub-bundles mentioned in section \ref{sec:sketch-construction}.
\begin{Proposition}
The vector bundle $E$ splits as $E = F_1 \oplus F_2$, a direct sum of two rank $2$ sub-bundles, such that $F_1$ has Euler class $|e_1 - e_2|$ and $F_2$ has Euler class $|e_1+e_2|$. Moreover, the direct sum is $Q$-orthogonal, $F_1$ is time-like and $F_2$ is space-like.
\end{Proposition}
\begin{proof}
We can write an isomorphism between $LN$ and $F_2$ as real vector bundles:
$$LN \ni v \rightarrow v+\tau v \in F_2$$
Hence the Euler class of $F_2$ is the same as the Euler class of $LN$, a complex line bundle of degree $(e_1+e_2)/2$. Similar argument works for $F_1$. The rest follows from an easy computation, since $Q$ and $\tau$ are both explicit.  
\end{proof}

Now we can properly define the circle bundle $U$ mentioned in the previous section:
$$U = \{ v \in F_1 \ |\ Q(v,v) = 1 \}. $$

\section{The pullback metrics and Domination conditions}   \label{sec:domination}

In this section we need to analyze more deeply the structure of the $\rho_i$-equivariant harmonic map $\psi_i$ from the universal cover of $\Sigma$ to the hyperbolic plane $\HH$. 

We will denote by $g_\HH$ the hyperbolic metric on $\HH$, with constant curvature $-1$, and by $\text{Vol}(g_\HH)$ its volume form. The $2$-tensors $\psi_i^* g_\HH$ and $\psi_i^*\text{Vol}(g_\HH)$ are $\pi_1(\Sigma)$-invariant, hence they descend to $2$-tensors on $\Sigma$ that we will denote by $g_i$ and $\text{Vol}(g_i)$. We will call $g_i$ the \defin{pull-back metric}, even if it is only a symmetric positive semi-definite $2$-tensor that can be degenerate at some points. 

We will call the anti-symmetric $2$-tensor $\text{Vol}(g_i)$ the \defin{pull-back volume form}, even if it is not in general a volume form, because it can have zeros and changes of sign. The Euler number $e_i$ of the representation  $\rho_i$ is given by (see \cite[sec.  3.5]{BIW}):
\begin{equation}\label{eulernumber}
  e_i=\frac{1}{2\pi}\int_{\Sigma} \text{Vol}(g_i)
\end{equation}

Now we want to express $g_i$ and $\text{Vol}(g_i)$ in terms of the Higgs bundle data. The following theorem is new only for non-Fuchsian representations. For Fuchsian representations in $SL(2,\R)$,  see \cite{HitchinSelfDuality} and \cite{WolfHarmonic}.

\begin{Theorem} \label{metricandvolume}
(1) The pull back metric $g_i$ is given by 
$$g_1=4\alpha\beta dz^2+4(k^2|\beta|^2+k^{-2}|\alpha|^2)dzd\bar z+4\bar\alpha\bar\beta d\bar z^2.$$
$$g_2=4\gamma\delta dz^2+4(h^2|\delta|^2+h^{-2}|\gamma|^2)dzd\bar z+4\bar\gamma\bar\delta d\bar z^2.$$
(2) The pull-back volume form $\text{\textup{Vol}}(g_i)$ is given by
$$\text{\textup{Vol}}(g_1)=4(k^{-2}|\alpha|^2-k^2|\beta|^2)dx\wedge dy.$$
$$\text{\textup{Vol}}(g_2)=4(h^{-2}|\gamma|^2-h^2|\delta|^2)dx\wedge dy.$$
\end{Theorem}
\begin{proof} We prove the theorem only for $\psi_1$, here denoted for simplicity by $\psi$. The case of $\psi_2$ is similar. Since our result is local, instead of working with the $\rho_1$-equivariant map $\psi$, we will consider the corresponding section of the flat bundle $\HH_{\rho_1}$, i.e. the metric $H_1$.  
But at every point $z \in \Sigma$, the metric $H_1$ is fixed by the real structure $\tau_1$, hence it is a symmetric, positive definite matrix on the real part $E_1$. The subset of the $SL(2,\C)/SU(2)$ fiber that is fixed by $\tau_1$ is a copy of $SL(2,\R)/SO(2)$, that is parallel for the flat structure.

In each fiber of $\HH_{\rho_1}$, the metric induced by the Killing form at a point $H$ is: 
$$\left<X,Y\right>:=2\tr(H^{-1}X H^{-1}Y),$$ 
which gives constant curvature $-1$.  

Along the section, we can choose an orthonormal frame of the form 
$$\left\{\sigma_1=\frac{1}{2}\begin{pmatrix}0&k\\k^{-1}&0\end{pmatrix}, \sigma_2=\frac{1}{2}\begin{pmatrix} 0&ki\\-k^{-1}i&0\end{pmatrix}\right\}.$$ 
The form $\text{Vol}(g_\HH)$ along the section has the property that $\text{Vol}(g_\HH)(\sigma_1,\sigma_2)=1.$

The Higgs field and the harmonic maps are related by $\psi_z=\phi_1, \psi_{\bar z}=\phi_1^*$, (see \cite{Corlette}).
Hence the pull-back metric of $\psi$ is
\begin{align*}
g_i&=\left<\psi_z,\psi_z\right>dz^2+2\left<\psi_z,\psi_{\bar z}\right>dzd\bar z+\left<\psi_{\bar z},\psi_{\bar z}\right>d\bar z^2\\
&=\left<\phi_1,\phi_1\right>dz^2+2\left<\phi_1,\phi^*_1\right>dzd\bar z+\left<\phi^*_1,\phi^*_1\right>d\bar z^2,
\end{align*}
Applying direct calculation, part (1) follows.

For part (2), we have
\begin{eqnarray*}
d\psi&=&\phi_1 dz+\phi_1^*d\bar z =\begin{pmatrix}0&\alpha\\\beta&0\end{pmatrix}dz+\begin{pmatrix}0&k^{2}\bar \beta\\k^{-2}\bar \alpha&0\end{pmatrix}d\bar z\\
&=&\begin{pmatrix}0&\alpha+k^{2}\bar \beta\\\beta+k^{-2}\bar \alpha&0\end{pmatrix}dx+i\begin{pmatrix}0&\alpha-k^{2}\bar \beta\\
\beta-k^{-2}\bar \alpha&0\end{pmatrix}dy.
\end{eqnarray*} 
Let $\alpha+k^2\bar\beta=a+bi,\alpha-k^2\bar \beta=c+di$, hence the pull-back volume form 
\begin{align*}
\psi^*\text{Vol}(g_\HH) &=\text{Vol}(g_\HH)(\psi_x,\psi_y)dx\wedge dy\\
&=\text{Vol}(g_\HH)\left(2ak^{-1}\sigma_1+2bk^{-1}\sigma_2,-2dk^{-1}\sigma_1+2ck^{-1}\sigma_2\right)dx\wedge dy\\
&=k^{-2}(4ac+4bd)dx\wedge dy=4k^{-2}\text{Re}((\alpha+k^2\bar\beta)(\bar\alpha-k^2\beta))dx\wedge dy \\
&=4(k^{-2}|\alpha|^2-k^2|\beta|^2)dx\wedge dy. \qedhere
\end{align*}
\end{proof}

We will say that the pull-back metric $g_1$ \defin{strictly dominates} the pull-back metric $g_2$ if the symmetric $2$-tensor $g_1-g_2$ is positive definite. In this case we will write $g_1 > g_2$. This condition depends on $\rho_1, \rho_2$, and on the choice of $\Sigma$.

\begin{Proposition}
Let $\rho_1,\rho_2$ be two representations such that $g_1>g_2$. Then $\rho_1$ is Fuchsian, $|e_1| = 2g-2$, the harmonic map $\psi_1$ is a diffeomorphism, $g_1$ is a hyperbolic metric, and $|e_2| < 2g-2$.
\end{Proposition}
\begin{proof}
If $g_1 > g_2$, then $g_1$ is positive definite. Hence $\psi_1$ is a local diffeomorphism. Since $\psi_1$ is also $\rho_1$-equivariant, it is the developing map of a hyperbolic structure on $S$ with holonomy $\rho_1$. Hence $\rho_1$ is a Fuchsian representation, and, by \cite{GoldmanComponents}, $|e_1| = 2g-2$. As developing map of a hyperbolic structure, $\psi_1$ is a diffeomorphism.

If $\rho_2$ were also Fuchsian, the condition $g_1>g_2$ would give two hyperbolic metrics on the same surface with different areas, which is impossible, hence $|e_2|<2g-2$.
\end{proof}

This condition of strict domination between the pull-back metrics is very related to another notion of domination between the two representations $\rho_1,\rho_2$: we say that $\rho_1$ \defin{strictly dominates} $\rho_2$ if there exists a $(\rho_1,\rho_2)$-equivariant map $f:\HH \rightarrow \HH$ with Lipschitz constant strictly smaller than $1$. The latter notion was introduced in \cite{SaleinExotiques} and it only depends on the two representations $\rho_1, \rho_2$. The following proposition comes from \cite{TholozanDeforming} and \cite{DeroinTholozan}.

\begin{Proposition} \label{prop:equivalencedomination}
Given two representations $\rho_1,\rho_2$ in $SL(2,\R)$, there exists a holomorphic structure $\Sigma$ such that $g_1>g_2$ if and only if $\rho_1$ strictly dominates $\rho_2$.  
\end{Proposition} 
\begin{proof}
If there exists $\Sigma$ such that $g_1>g_2$, we have seen in the previous proposition that $\psi_1$ is invertible. The map $\psi_2 \circ \psi_1^{-1}:\HH \rightarrow \HH$ is $(\rho_1,\rho_2)$-equivariant and has Lipschitz constant strictly smaller than $1$.

The converse is the only interesting implication. It is proved in \cite{TholozanDeforming} that if $\rho_1$ strictly dominates $\rho_2$, there exists a unique holomorphic structure $\Sigma$ such that the corresponding Higgs bundles have $\alpha\beta = \gamma\delta$. In this case, it is proved in \cite{DeroinTholozan} that the difference of the pull-back metrics is positive definite.
\end{proof}

Given representations $\rho_1,\rho_2$ such that $\rho_1$ strictly dominates $\rho_2$, the set of holomorphic structures on $S$  such that $g_1 > g_2$ is an open subset of the Teichm\"uller space that plays an important role in this work. This open subset has a special point, the unique structure found by Tholozan in \cite{TholozanDeforming} such that $\alpha\beta = \gamma\delta$ (in our terminology, vanishing Pfaffian). In section \ref{sec:fibers} we will show how to give a geometric interpretation to this structure as the conformal structure of an $\rho_1\otimes\rho_2$-equivariant minimal surface in the Klein quadric.

\section{Construction of AdS 3-manifolds}  \label{sec:construction}

Following the plan sketched in section \ref{sec:sketch-construction}, after properly defining the circle bundle $U$ in section \ref{sec:higgsandflat}, we need to analyze the tautological section
$s: U \rightarrow \M_{\bar{\rho}}$.

In this section we will prove that, given two representations $\rho_1,\rho_2$ and a holomorphic structure $\Sigma$ such that $g_1 > g_2$, the section $s$ is a transverse section. This proves that there exists an AdS structure with holonomy $\overline{\rho_1\otimes \rho_2}$. 

When this is put together with proposition \ref{prop:equivalencedomination}, it gives a new proof of the following theorem, originally proved by Salein in \cite{SaleinExotiques}, and later reproved in \cite{GueritaudKassel}, which is one of the main theorems in the theory of AdS $3$-manifolds:

\begin{Theorem} \label{thm:geometricstructure}
Let $S$ be a closed surface. Given two representations $\rho_1,\rho_2$ of $\pi_1(S)$ in $SL(2,\R)$ such that $\rho_1$ strictly dominates $\rho_2$, there exists an AdS structure with holonomy $\overline{\rho_1\otimes\rho_2}$ on a circle bundle $U$ over $S$ with Euler class $|e_1-e_2|$.
\end{Theorem}

The result obtained by Salein gives more information than this theorem, we will recover that later, in section \ref{sec:discontinuous}.
With theorem \ref{thm:geometricstructure} we can construct AdS structures on all the circle bundles with even Euler class between $2$ and $2g-2$.  Our proof of the theorem gives a way to express the AdS metric in terms of explicit coordinates on the manifold $U$. Using this we can easily see some properties of these structures, as we will see in the next sections. The proof will take the rest of this section.

We need to write the equations of $U$ as a subset of $E$. Let $v\in U$, we have
\begin{equation*}v=\begin{pmatrix}0\\\nu\\\omega\\0\end{pmatrix}, \hspace{1cm} 
\tau v = \begin{pmatrix}0&0&0&hk\\0&0&h^{-1}k&0\\0&hk^{-1}&0&0\\h^{-1}k^{-1}&0&0&0\end{pmatrix}\begin{pmatrix}0\\\bar\nu\\\bar\omega\\0\end{pmatrix}=\begin{pmatrix}0\\h^{-1}k\bar\omega\\hk^{-1}\bar\nu\\0\end{pmatrix}. \end{equation*}
The reality condition requires that $v=\tau v$, therefore $\omega=hk^{-1}\bar\nu.$ The $Q$-norm 1 condition requires that $Q(v,v)=1.$ Hence $\nu\omega=1,$ and then $|\nu|=({hk^{-1}})^{-\frac{1}{2}}$.  

We give a local coordinate description of the tautological section $s$. We can cover $\Sigma$ with little open sets $V_i$ bi-holomorphic to the unit disk $\{|z|<1\}$. Locally over every $V_i$, $U$ is $\{|z|<1\}\times S^1.$ Denote the function $g=({hk^{-1}})^{-\frac{1}{2}}$. Given coordinates $(z,\theta)$, we write $v(z,\theta)=\begin{pmatrix}0\\g e^{i\theta}\\ g^{-1}e^{-i\theta}\\0\end{pmatrix}\in U.$ 

Then the section $s$ at $v$ and its derivatives are given by
\begin{eqnarray*}
&s(v)=\begin{pmatrix}0\\g e^{i\theta}\\ g^{-1}e^{-i\theta}\\0\end{pmatrix}, \ \ \ \ &\nabla_{\frac{\partial}{\partial\theta}}s=\begin{pmatrix}0\\ige^{i\theta}\\-ig^{-1}e^{-i\theta}\\0\end{pmatrix},\\
&\nabla_{\frac{\partial}{\partial  z}}s=\begin{pmatrix}X\\0\\0\\Y\end{pmatrix}+c\nabla_{\frac{\partial}{\partial  \theta}}s, \ \ \ \ &\nabla_{\frac{\partial}{\partial\bar  z}}s=\begin{pmatrix}Z\\0\\0\\W\end{pmatrix}+\bar c\nabla_{\frac{\partial}{\partial  \theta}}s.\end{eqnarray*}
where $c=ig^{-1}\partial g,$ and

$$\begin{array}{ll}
X=\gamma ge^{i\theta}+\alpha g^{-1} e^{-i\theta}, \hspace{1cm} &Y=\beta ge^{i\theta}+\delta g^{-1} e^{-i\theta},\\
Z=h^2\bar \delta ge^{i\theta}+k^2\bar\beta g^{-1}e^{-i\theta}, &W=k^{-2}\bar \alpha g e^{i\theta}+h^{-2}\bar \gamma g^{-1}e^{-i\theta}.
\end{array}$$

\begin{Lemma}\label{XZindependence}
Suppose that the pull-back metrics satisfy $g_1>g_2$.
Then, if $u\in \C$ satisfies $uX+\bar uZ=0$, we have $u=0.$
\end{Lemma}
\begin{proof} We carry out the following direct computation:
\begingroup
\allowdisplaybreaks
\begin{align*}
&\ \ \ \ \ \  uX+\bar uZ=0\\
& \Longrightarrow u (\gamma ge^{i\theta}+\alpha g^{-1}e^{-i\theta})+\bar u(h^2\bar \delta ge^{i\theta}+k^2\bar\beta g^{-1}e^{-i\theta})=0\\
&\Longrightarrow|u\gamma g e^{i\theta}+\bar uh^2\bar \delta g e^{i\theta}|^2=|u\alpha g^{-1}e^{-i\theta}+\bar u k^2\bar\beta g^{-1} e^{-i\theta}|^2\\
&\Longrightarrow|uh^{-1}\gamma+\bar uh\bar \delta |^2=|uk^{-1}\alpha+\bar u k\bar\beta|^2, \ \ \  \text{since\ \ $g=({hk^{-1}})^{-\frac{1}{2}}$}\\
&\Longrightarrow \gamma\delta u^2+|u|^2(h^{-2}|\gamma|^2+h^2|\delta|^2)+\bar\delta\bar\gamma \bar u^2=\\
& \ \ \ \ \ \ \ =\alpha\beta u^2+|u|^2(k^{-2}|\alpha|^2+k^2|\beta|^2)+\bar\alpha\bar\beta \bar u^2\\
&\text{Applying Proposition \ref{metricandvolume}}\\
& \Longrightarrow g_1\left(u\frac{\partial}{\partial z},\bar u\frac{\partial}{\partial \bar z}\right)=g_2\left(u\frac{\partial}{\partial z},\bar u\frac{\partial}{\partial\bar z}\right) \Longrightarrow u=0, \ \ \ \ \text{since}\ \  g_1>g_2.\qedhere
\end{align*}
\endgroup
\end{proof}

The transversality condition for the section $s$ is equivalent to the condition that at every point $(x,y,\theta) \in U$, the derivatives $\nabla_{\frac{\partial}{\partial x}} s, \nabla_{\frac{\partial}{\partial y}} s, \nabla_{\frac{\partial}{\partial \theta}} s$ form a basis of the tangent space at $s(x,y,\theta)$ to the fiber of $\M_{\bar{\rho}}$. Equivalently, we need to check that the four vectors $s, \nabla_{\frac{\partial}{\partial x}} s, \nabla_{\frac{\partial}{\partial y}} s, \nabla_{\frac{\partial}{\partial \theta}} s$ are linearly independent. 

\begin{Theorem}
If the pull-back metrics satisfy $g_1>g_2$, then the section $s$ is a transverse section.
\end{Theorem}
\begin{proof}
Suppose that there exists $a, b, c, d\in \R$ such that \begin{equation*}a\nabla_{\frac{\partial}{\partial  x}}s+b\nabla_{\frac{\partial}{\partial y}}s+c\nabla_{\frac{\partial}{\partial \theta}}s+ds=0.\end{equation*} 

Written $\nabla_{\frac{\partial}{\partial  z}}s, \nabla_{\frac{\partial}{\partial  \bar z}}s$ instead of $\nabla_{\frac{\partial}{\partial  x}}s, \nabla_{\frac{\partial}{\partial  y}}s$, we have 
\begin{equation*}(a+bi)\nabla_{\frac{\partial}{\partial  z}}s+(a-bi)\nabla_{\frac{\partial}{\partial \bar z}}s+c\nabla_{\frac{\partial}{\partial \theta}}s+ds=0.\end{equation*}

Taking the first entry of the vector of $(a+bi)\nabla_{\frac{\partial}{\partial  z}}s+(a-bi)\nabla_{\frac{\partial}{\partial \bar z}}s+c\nabla_{\frac{\partial}{\partial \theta}}s+ds$,
\begin{equation}
(a+bi)X+(a-bi)Z=0.
\end{equation}
By Lemma \ref{XZindependence}, we obtain that $a=b=0$.

Therefore $c\nabla_{\frac{\partial}{\partial \theta}}s+ds=0$. It is clear that $c=d=0.$
\end{proof}

\section{Minimal immersions and fibers of the circle bundle}  \label{sec:fibers}

Using the above construction, not only we construct the AdS 3-manifold, but also recover properties of the circle fibers of $U$. Let $z\in \Sigma$, and let $C=p^{-1}(c)$, the fiber of $z$ in $U$, topologically a circle. By fixing a base point $c \in C$, the circle $C$ becomes a loop in $U$. Let $\widetilde{C}$ be a lift of this loop to the universal covering $\widetilde{U}$. The image $D(\widetilde{C})$ of $\widetilde{C}$ by the developing map $D$, is a curve in $\M$, that we will call the developing image of $C$. In the next theorem we describe the curve $D(\widetilde{C})$. A similar theorem was proved by Gu\'eritaud and Kassel in \cite{GueritaudKassel}.

\begin{Theorem} \label{thm:fibers}
The developing image $D(\widetilde{C})$ is a time-like geodesic loop that turns once around $\M$.
\end{Theorem}
\begin{proof}
We can compute the tangent vector along the fiber and its derivative:
$$s = \begin{pmatrix}0\\ge^{i\theta}\\g^{-1}e^{-i\theta}\\0\end{pmatrix} \hspace{0.7cm} \nabla_{\frac{\partial}{\partial\theta}}s=\begin{pmatrix}0\\ige^{i\theta}\\-ig^{-1}e^{-i\theta}\\0\end{pmatrix}  \hspace{0.7cm}   \nabla_{\frac{\partial}{\partial\theta}}\nabla_{\frac{\partial}{\partial\theta}}s=-\begin{pmatrix}0\\ge^{i\theta}\\g^{-1}e^{-i\theta}\\0\end{pmatrix}=-s $$
From the second formula we see that $Q(\nabla_{\frac{\partial}{\partial\theta}}s,\nabla_{\frac{\partial}{\partial\theta}}s)=1>0$, hence the fiber is time-like, from the third formula we see that it is a geodesic. The flat connection we have on the bundle $p^*E$ is the pull back of a flat connection on $E$, hence it is trivial on the fiber. So the formula given for $s$ gives an explicit parametrization of $D(\widetilde{C})$ in $\M$, and we can see that it is a loop and it turns around once.
\end{proof}

We can use the structure of the circle fibers of an AdS structure on $U$ to construct and characterize minimal immersions of Riemann surfaces into quadrics. This is a new result of this paper. The bundle $E$ over $\Sigma$ has structure group $O(2,2)$, and this group also acts on $\text{Gr}(2,4)$, the Grassmannian of $2$-planes in $\R^4$. By changing fiber, the bundle $E(\text{Gr}(2,4))$ with fiber $\text{Gr}(2,4)$ and same structure group inherits also a flat structure. Every circle fiber of $U$ corresponds to a time-like geodesic, i.e. to a $2$-plane in the corresponding fiber of $E$. This gives a section of $E(\text{Gr}(2,4))$. As in section \ref{sec:geometricstructures}, such a section induces a $\rho_1\otimes \rho_2$-equivariant map
$f:\widetilde{\Sigma} \rightarrow \text{Gr}(2,4).$

The map $f$ takes values in the time-like part of $\text{Gr}(2,4)$, here denoted by $\text{TGr}(2,4)$ which is a homogeneous space for $O(2,2)$. The \defin{Pl\"ucker embedding} 
$$\text{Gr}(2,4) \ni \text{Span}(v,w) \rightarrow v\wedge w \in \P\left(\Lambda^2 \R^4\right) $$
identifies the Grassmannian with a hypersurface of a projective space, usually called the \defin{Klein quadric}, with quadratic equation
$$\text{Gr}(2,4) = \{ [v] \in \P\left(\Lambda^2 \R^4\right) \ |\ v \wedge v = 0 \} $$
The wedge product induces a symmetric bilinear form $v \wedge w$ on $\Lambda^2 \R^4 $ with signature $(3,3)$. This descends to a signature $(2,2)$ pseudo-Riemannian metric on $\text{Gr}(2,4)$ which is preserved by the action of $O(2,2)$. We remark that, by a quite general result about real projective structures (see \cite[Lemma 3.5.0.1]{BaragliaThesis}) the differential of $f$ has always image in a definite subspace of the tangent space to $\text{Gr}(2,4)$.

In the next theorem we will prove that the map $f$ is always harmonic, and, if $\Sigma$ is the holomorphic structure with vanishing Pfaffian (see section \ref{sec:domination}), the map $f$ is also conformal, hence it is a minimal immersion. This gives a way to construct $\rho_1\otimes \rho_2$-equivariant minimal immersions in the Klein quadric, and a geometric interpretation of the holomorphic structure with vanishing Pfaffian. In the last theorem of the section, we will show that representations in $O(2,2)$ admitting an equivariant minimal immersion are all associated with some AdS structure, obtaining in this way a complete characterization. 

This relationship between conformality and the Pfaffian is in contrast with the case for $Sp(4,\R)$, studied by Baraglia \cite{BaragliaThesis}, where the conformality comes from vanishing of the Hopf differential. 

Now we prove the claims. Locally, the map $f$ can be written as $f=s\wedge \nabla_{\frac{\partial}{\partial\theta}}s$. In the following, for simplicity, we will denote $\nabla_{\frac{\partial}{\partial\theta}}s,\nabla_{\frac{\partial}{\partial z}}s,\nabla_{\frac{\partial}{\partial\bar z}}s$ by $s_{\theta}, s_z,s_{\overline{z}}$.

\begin{Theorem}     \label{thm:harmimm}
The map $f$ is harmonic and it is conformal if and only if $\alpha\beta=\gamma\delta.$ In this case $f$ is a $\rho_1 \otimes \rho_2$-equivariant minimal immersion in the Klein quadric with image in $\text{TGr}(2,4)$.
\end{Theorem}
\begin{proof}
For the map from surface, the condition that $f$ is harmonic is equivalent to $\widetilde{\nabla}_{\bar{z}} f_z = 0$, where $\widetilde{\nabla}$ is the Levi-Civita connection on $\text{Gr}(2,4)$ (see \cite{Helein_harmonicmaps} page 425). We can compute $f_z$: 
\[f_z=(s\wedge s_{\theta})_z=s_z\wedge s_{\theta}+s\wedge s_{\theta,z},\]

Differentiating $s_z$ with respect to $\theta$, $s_{z,\theta}=\left( \begin{array}{c}
i\gamma g e^{i\theta}-i\alpha g^{-1}e^{-i\theta} \\
0\\
0\\
i\beta ge^{i\theta}-i\delta g^{-1}e^{-i\theta} \end{array} \right)-c s,$ using the formula $\nabla_{\frac{\partial}{\partial\theta}}\nabla_{\frac{\partial}{\partial\theta}}s=-s.$

In the coordinates given by the six minors of a $2\times 4$ matrix, we have
$$s_z\wedge s_{\theta} =  (i\gamma g^2 e^{2 i\theta } + i\alpha, -i\gamma - i\alpha g^{-2} e^{-2i\theta }, 0, 0, -i\beta g^2 e^{2i\theta } -i\delta, i\beta+ i\delta g^{-2} e^{-2i\theta })$$
$$s\wedge s_{\theta,z} =  (-i\gamma g^2 e^{2 i\theta } + i\alpha, -i\gamma + i\alpha g^{-2} e^{-2i\theta }, 0, 0, i\beta g^2 e^{2i\theta } -i\delta, i\beta- i\delta g^{-2} e^{-2i\theta })$$

$f_z = (2i\alpha, -2i\gamma, 0, 0, -2i\delta, 2i\beta)$, and since it is holomorphic, $\widetilde{\nabla}_{\bar{z}} f_z = 0$.

The condition that the map $f$ is conformal is equivalent to $\left<f_z,f_z\right>=0$, where the pairing is the pseudo-Riemannian structure on $Gr(2,4)$. 
 \[\left<f_z,f_z\right>=f_z\wedge f_z=2s_z\wedge s_{\theta}\wedge s\wedge s_{\theta,z} \]
\[
=2\det\left( \begin{array}{cccc}
\gamma g e^{i\theta}+\alpha g^{-1}e^{-i\theta}&0&0&i\gamma g e^{i\theta}-i\alpha g^{-1}e^{-i\theta} \\
0&ige^{i\theta}&ge^{i\theta}&0\\
0&-ig^{-1}e^{-i\theta}&g^{-1}e^{-i\theta}&0\\
\beta ge^{i\theta}+\delta g^{-1}e^{-i\theta}&0&0&
i\beta ge^{i\theta}-i\delta g^{-1}e^{-i\theta} \end{array} \right)d\text{Vol}_{{\R}^4}\]
\[\
=-8(\alpha\beta-\gamma\delta)d\text{Vol}_{{\R}^4}.\]

Hence the map $f$ is conformal if and only if $\alpha\beta=\gamma\delta.$ 
\end{proof}

\begin{Theorem}  \label{thm:minsurf}
Let $S$ be a closed surface and $\rho:\pi_1(S) \rightarrow O(2,2)$ be a representation. Then there exists a $\rho$-equivariant minimal immersion $f:\widetilde{S} \rightarrow \text{TGr}(2,4)$ if and only if there exists a finite covering $p:S'\rightarrow S$ and representations $\rho_1,\rho_2:\pi_1(S') \rightarrow SL(2,\R)$ such that $\rho \circ p_* = \rho_1 \otimes  \rho_2$ and $\rho_1$ strictly dominates $\rho_2$.  

Moreover, the conformal structure of this minimal surface is equivalent to the special holomorphic structure on $S'$ with vanishing Pfaffian found in \cite{TholozanDeforming}.
\end{Theorem}
\begin{proof}
The ``if'' part was proven in the previous theorem. To see the ``only if'' part, we notice that the conformality condition guaranties that the image of the differential of the map lies in the positive definite or negative definite part of the tangent space to $\text{TGr}(2,4)$. By \cite[Proposition 3.5.1]{BaragliaThesis}, $f$ induces a real projective structure on a circle bundle over $S$, with holonomy $\rho$. The developing image of this structure is contained in the model space $\M$, because all the time-like geodesics are entirely contained in $\M$. Since the holonomy is in $O(2,2)$, this real projective structure is actually an AdS structure. As explained in the introduction, up to a finite covering, this structure is a quotient of $\M$ by the action of a tensor product $\rho_1\otimes\rho_2$. By Kassel's theorem \cite{KasselThese}, $\rho_1$ strictly dominates $\rho_2$.
\end{proof}

\section{Properly discontinuous actions}  \label{sec:discontinuous}

In this section we will recover the theorem proved in \cite{SaleinExotiques} and later in \cite{GueritaudKassel}, here as theorem \ref{thm:salein}. This is interesting as an example where the theories of Higgs bundles and of geometric structures make it possible to understand the dynamical properties of representations: in this case we will show that if $\rho_1$ strictly dominates $\rho_2$, the representation $\rho_1 \otimes \rho_2$ acts freely and properly discontinuously on $\M$.
 
We need to remark that the model space $\M$ we have chosen for AdS geometry is not simply connected, it is a $3$-dimensional cylinder $\R^2 \times S^1$, hence $\pi_1(\M) = \Z$. Consider $\widetilde{\M}$, the universal covering of $\M$, with its induced AdS structure. Following \cite{KulkarniRaymond}, the covering $\widetilde{\M} \rightarrow \M$ is equivariant for the natural homomorphism:
$$p:\text{Isom}_0(\widetilde{\M}) =  \widetilde{SL(2,\R)} \times \widetilde{SL(2,\R)}/ Z( \widetilde{SL(2,\R)}) \rightarrow$$
$$\rightarrow SL(2,\R) \times SL(2,\R)/ \Z_2 = SO_0(2,2)$$ 

\begin{Theorem} \label{thm:salein}
Let $S$ be a closed surface. Given two representations $\rho_1,\rho_2$ of $\pi_1(S)$ in $SL(2,\R)$ such that $\rho_1$ strictly dominates $\rho_2$, the representation $\rho_1\otimes \rho_2$ acts freely and properly discontinuously on $\M$, and the quotient $\M / \rho_1\otimes \rho_2$ is the AdS structure we constructed on $U$ in section \ref{sec:construction}. 
\end{Theorem}
\begin{proof}
By the result of Klingler \cite{KlinglerCompletude}, every AdS structure on a closed manifold is \defin{complete}, i.e., the developing map $\widetilde{U} \rightarrow \M$ is a covering. Hence it lifts to a map $\widetilde{U} \rightarrow \widetilde{\M}$ which is an isomorphism, and the action of $\pi_1(U)$ on $\widetilde{U}$ by deck transformations induces a representation $\sigma:\pi_1(U) \rightarrow \text{Isom}_0(\widetilde{\M})$. By construction of $\sigma$ we have that $\widetilde{\M}/\sigma(\pi_1(U))$ is isomorphic to $U$: 
$\widetilde{\M} \rightarrow  \widetilde{\M}/\sigma(\pi_1(U)) = U$. 

We will show that the model space $\M = \widetilde{\M}/\text{Deck}(\widetilde{\M})$ is an intermediate covering, in other words we need to show that $\text{Deck}(\widetilde{\M})<\sigma(\pi_1(U))$.

The covering $\widetilde{\M} \rightarrow \M$ is equivariant for the projection $p:\text{Isom}_0(\widetilde{\M}) \rightarrow SO_0(2,2)$ described above, hence the holonomy representation $\overline{\rho_1\otimes\rho_2}$ is exactly $p\circ\sigma$.

Let $c\in \pi_1(U)$ be the class of the circle fiber. We have seen in theorem \ref{thm:fibers} that the developing image of the circle fiber of $U$ is a loop in $\M$, that turns once around $\M$, hence $\sigma(c)\in\text{Isom}_0(\widetilde{\M})$ is a generator of the group $\text{Deck}(\widetilde{\M})$. So we have:
$$\widetilde{\M} \rightarrow \widetilde{\M}/ \langle \sigma(c)\rangle = \M \rightarrow  \widetilde{\M}/ \sigma(\pi_1(U)) = U$$
Then $\M / \rho_1\otimes\rho_2 = U$, hence $\rho_1\otimes\rho_2$ acts freely and properly discontinuously. 
\end{proof}

\section{Volume of AdS 3-manifolds}  \label{sec:volumes}

The computation of the volume of closed AdS 3-manifolds was  Question 2.3 in the list of open problems \cite{BBDGGKKSZ}. It was first answered some months ago in Nicolas Tholozan's thesis \cite{TholozanThese}, see also \cite{TholozanVolumes}. Labourie then related the volumes of AdS $3$-manifolds with the Chern-Simons invariants. Here we give a different proof of the formula for the volume, based on our construction of the AdS structures.

\begin{Theorem}
The volume of $\M / \rho_1\otimes\rho_2$ is $\pi^2 |e_1+e_2|$.
\end{Theorem}
Note here $|e_1| = 2g-2$. This formula differs from the one in  \cite[Lemma 4.5.2]{TholozanThese} by a factor $2$ because the model space used there is a quotient of $\M$ by $\Z/2\Z$. 
\begin{proof}
Denote $G$ as the matrix presentation of the Lorentzian metric tensor in terms of the basis $s_z,s_{\bar z},s_{\theta}$. The volume form $d\text{Vol}=|\sqrt{|\det G|}dz\wedge d\bar z\wedge d\theta|$.

By direct calculation, we see that $Q(s_{\theta},s_{\theta})=1$, $Q(s_z,s_{\theta})=c$, $Q(s_{\bar z},s_{\theta})=\bar c$, $Q(s_z,s_z)=-XY+c^2$, $Q(s_{\bar z},s_{\bar z})=-ZW+{\bar c}^2$, $Q(s_z,s_{\bar z})=-\frac{1}{2}(XW+YZ)+c\bar c$.
\vspace{0.2cm}\\
Hence $G=\begin{pmatrix}-XY+c^2&-\frac{1}{2}(XW+YZ)+c\bar c&c\\-\frac{1}{2}(XW+YZ)+c\bar c&-ZW+{\bar c}^2&\bar c\\c&\bar c&1\end{pmatrix}$, and then

\begin{equation*}
\text{det} (G)=XY\cdot ZW-\frac{1}{4}(XW+YZ)^2=-\frac{1}{4}(XW-YZ)^2.
\end{equation*}

Therefore the volume of AdS 3-manifold is 
\vspace{0.1cm}\\ \noindent
$\displaystyle\int_U d\text{Vol}=\frac{1}{2}\int_U |XW-YZ|dz\wedge d\bar z\wedge d\theta$\vspace{0.1cm}=\\
The non-degeneracy of volume form implies $XW-YZ$ has constant sign on $U$.\vspace{0.1cm}\\
$\displaystyle\hphantom{\int_U d\text{Vol}}=\frac{1}{2}\left|\int_U (XW-YZ)dz\wedge d\bar z\wedge d\theta\right|=$\vspace{0.1cm}\\
$\displaystyle\hphantom{\int_U d\text{Vol}}=\frac{1}{2}\left|\int_{\Sigma}\int_{S^1}((\gamma\bar \alpha g^2e^{2i\theta}k^{-2}+|\alpha|^2k^{-2}+|\gamma|^2h^{-2}+\bar \gamma \alpha h^{-2}g^{-2} e^{-2i\theta})\right.$\vspace{0.1cm}\\
$\displaystyle\hphantom{\int_U d\text{Vol}=\ \ }\left.\vphantom{\int}-(\beta\bar\delta g^2 e^{2i\theta}h^2+|\delta|^2h^2+|\beta|^2k^2+\delta\bar\beta k^2 g^{-2}e^{-2i\theta}))dz\wedge d\bar z\wedge d\theta\right|=$\vspace{0.1cm}\\
$\displaystyle\hphantom{\int_U d\text{Vol}}=2\pi\left|\int_{\Sigma}((|\alpha|^2k^{-2}-|\beta|^2k^2)+(|\gamma|^2h^{-2}-|\delta|^2h^2))dx\wedge dy\right|=$\vspace{0.1cm}\\
Applying theorem \ref{metricandvolume} and equation (\ref{eulernumber}),
\vspace{0.1cm}\\
$\displaystyle\hphantom{\int_U d\text{Vol}}=\frac{1}{2}\pi\left|\int_{\Sigma}\text{Vol}(g_1)+\text{Vol}(g_2)\right| = \pi^2 |e_1+e_2|$.
\end{proof}

\end{document}